\documentclass[11pt,a4paper,oneside]{article}
\newcommand{\ignore}[1]{}
\usepackage{graphicx} 
\usepackage{amsmath}
\usepackage{amssymb}
\usepackage{amsthm}

\usepackage{color}

\newtheorem{thm}{Theorem}[section]\theoremstyle{plain}
\newtheorem{theorem}[thm]{Theorem}\theoremstyle{plain}
\theoremstyle{plain}
\newtheorem{lemma}[thm]{Lemma}\theoremstyle{plain}
\theoremstyle{plain}

\theoremstyle{plain}
\newtheorem{claim}[thm]{Claim}\theoremstyle{plain}
\newtheorem{corollary}[thm]{Corollary}\theoremstyle{plain}
\theoremstyle{plain}
\theoremstyle{plain}
\theoremstyle{plain}
\theoremstyle{plain}
\newtheorem{conjecture}[thm]{Conjecture}\theoremstyle{plain}
\theoremstyle{plain}

\theoremstyle{definition}
\theoremstyle{plain}

\theoremstyle{definition}
\theoremstyle{plain}

\newcommand{\RR}{{\mathbb R}}

\newcommand{\rank}{{\rm rank }}
\newcommand{\R}{{\mathbb R}}

\newcommand{\glc}{{\rm glc}}
\newcommand {\cR}{{\cal R}}

\title{\bf Characterizing globally linked pairs in 
graphs}

\author{
Tibor Jord{\'a}n\thanks{Department of Operations Research,
ELTE E\"otv\"os Lor\'and University, and the 
HUN-REN-ELTE Egerv\'ary Research Group on Combinatorial Optimization,
P\'azm\'any P\'eter s\'et\'any 1/C, 1117 Budapest, Hungary.
e-mail: {\tt tibor.jordan@ttk.elte.hu}} 
\and
Shin-ichi Tanigawa
\thanks{Department of Mathematical Informatics, Graduate School of Information Science and Technology, University of Tokyo, 7-3-1 Hongo, Bunkyo-ku, 113-8656, Tokyo, Japan.
e-mail: {\tt tanigawa@mist.i.u-tokyo.ac.jp}}
}

\date{March 26, 2026}

\begin{document}

\maketitle

	\begin{abstract}
    A pair $\{u,v\}$ of vertices is said to be globally linked in 
    a $d$-dimensional framework $(G,p)$ if there exists no other
    framework $(G,q)$ with the same edge lengths, in which the
    distance between the points corresponding to $u$ and $v$ 
    is different from that in $(G,p)$. 
    We say that $\{u,v\}$ is globally linked in $G$ in $\R^d$ if
    $\{u,v\}$ is globally linked in every generic $d$-dimensional framework $(G,p)$.

    We give a complete combinatorial characterization of globally linked
    vertex pairs in graphs in $\R^2$, solving a 
    conjecture of Jackson, Jord\'an and Szabadka from 2006 in the affirmative.
    Our result provides a refinement of the characterization of globally rigid graphs in $\R^2$ as well as an efficient algorithm for finding the globally linked pairs in a graph. We can also deduce that globally linked pairs in $\R^2$, globally linked pairs in ${\mathbb C}^2$, and stress-linked pairs in ${\mathbb R}^2$ are all the same,
    settling conjectures of Jackson and Owen, and Garamv\"olgyi, respectively.
    In higher dimensions we determine the 
    globally linked pairs in body-bar graphs in $\R^d$, for all $d\geq 1$, verifying
    a conjecture of Connelly, Jord\'an and Whiteley.
	\end{abstract}
	

\section{Introduction}

We briefly introduce the basic notions of combinatorial rigidity theory. 
Let $d \geq 1$ be an integer. A \emph{(bar-and-joint) framework} in $\RR^d$ is a pair $(G,p)$, where $G = (V,E)$ is a graph and $p : V \rightarrow \RR^d$ is a function that maps the vertices of $G$ into Euclidean space. We also say that $(G,p)$ is a \emph{realization} of $G$ in $\RR^d$. Two realizations $(G,p)$ and $(G,q)$ are \emph{equivalent} if the edge lengths coincide in the two frameworks, that is, if $||p(u) - p(v)|| = ||q(u) - q(v)||$ for every edge $uv \in E$, where $||.||$ denotes the Euclidean norm in $\R^d$. 
The realizations are \emph{congruent} if $||p(u) - p(v)|| = ||q(u) - q(v)||$ holds for all pairs of vertices $u,v \in V$. A framework $(G,p)$ in $\RR^d$ is \emph{globally rigid} if every equivalent framework $(G,q)$ in $\RR^d$ is congruent to $(G,p)$. As a local counterpart, we define $(G,p)$ to be \emph{rigid} if there is some $\varepsilon > 0$ such that every equivalent framework $(G,q)$ in $\RR^d$ such that $||p(v) - q(v)|| < \varepsilon$ for all $v \in V$ is congruent to $(G,p)$. 

A framework $(G,p)$ is \emph{generic} if the (multi)set of coordinates of $p(v),v \in V$ is algebraically independent over $\mathbb{Q}$. It is known that for a given dimension $d \geq 1$, the rigidity and global rigidity of generic realizations of $G$ in $\RR^d$ are determined by $G$ itself, see \cite{AR,Con,GHT}.
We say that $G$ is \emph{rigid} in $\RR^d$ (or {\it $d$-rigid}, for short)
if every, or equivalently, if some generic realization of $G$ in $\RR^d$ is rigid.
Similarly, we say that $G$ is
\emph{globally rigid} in $\RR^d$ (or {\it globally $d$-rigid})
if every, or equivalently, if some generic realization of $G$ in $\RR^d$ is globally rigid.

It follows from the definitions  that globally rigid graphs are rigid. 
The following stronger necessary conditions of global rigidity 
 are due to Hendrickson \cite{hend}.
We say that a graph is \textit{redundantly rigid} in $\mathbb{R}^d$ if it
remains rigid in $\RR^d$ after deleting any edge. A graph is \emph{$k$-connected} for some positive integer $k$ if it has at least $k+1$ vertices and it remains connected after
deleting any set of less than $k$ vertices.

\begin{theorem}\cite{hend}\label{theorem:hendrickson}
Let $G$ be a graph on $n \geq d+2$ vertices for some $d \geq 1$. Suppose that $G$ is globally rigid in $\RR^d$. Then $G$ is $(d+1)$-connected and redundantly rigid in $\RR^d$.
\end{theorem}

For $d=1,2$ the conditions of Theorem \ref{theorem:hendrickson} are, in fact, sufficient for global rigidity. 
It is well-known that a graph is globally rigid in $\R^1$ if and only if it is $2$-connected (see e.g. \cite[Theorem 63.2.6]{JW}).
The characterization of $2$-dimensional global rigidity is as follows. 
The notion of ${\cal R}_2$-connectivity (precisely defined in the next section) is central in the theory of global rigidity. 
Roughly speaking, a graph $G$ has this property if 
an associated matroid, defined on the edge set of $G$, cannot be
written as the direct sum of two matroids.

\begin{theorem} \cite{JJconnrig}
\label{theorem:2dimgloballyrigid}
Let $G$ be a graph on at least four vertices.
The following assertions are equivalent.\\
(a) $G$ is globally rigid in $\R^2$,\\
(b) $G$ is $3$-connected and redundantly rigid in $\R^2$,\\
(c)
$G$ is $3$-connected and $\mathcal{R}_2$-connected.
\end{theorem}

For $d\geq 3$ the conditions of Theorem \ref{theorem:hendrickson}, together, are no longer
sufficent to imply global rigidity and the combinatorial characterization of globally rigid graphs in these dimensions is a major open question. 

In this paper we shall focus on a refinement of global rigidity, defined as follows.
Let $(G,p)$ be a framework in $\RR^d$. 
Following \cite{JJS}, we define a pair of vertices $\{u,v\}$ to be \emph{globally linked} in $(G,p)$ if for every equivalent framework $(G,q)$ in $\RR^d$, we have $||p(u) - p(v)|| = ||q(u) - q(v)||$. Thus, $(G,p)$ is globally rigid if and only if every pair of vertices is globally linked in $(G,p)$. In contrast with global rigidity, this is not, in general, a generic property: it may happen that (in a given dimension $d\geq 2$) $\{u,v\}$ is globally linked in some generic realizations and not globally linked in others, see \cite[Figs. 1 and 2]{JJS}. We define the pair $\{u,v\}$ to be \emph{globally linked in $G$ in $\RR^d$} (or
{\it globally $d$-linked}) if $\{u,v\}$ is globally linked in {\em every} generic realization of $G$ in $\RR^d$. 

One may also consider vertex pairs which are
globally linked in {\em some} generic realization of $G$ in $\R^d$.
These pairs are called {\it weakly globally $d$-linked} in $G$.
A combinatorial characterization of weakly globally 2-linked pairs
can be found  in \cite{JV}.

Global linkedness is well-understood for $d=1$. For a graph $G$ and two vertices $u,v\in V(G)$ we use $\kappa_G(u,v)$ to denote the maximum number of pairwise
internally vertex-disjoint $u$-$v$ paths in $G$. It can be shown that $\{u,v\}$ is globally 1-linked in $G$ if and only if $uv$ is an edge of $G$ or $\kappa_G(u,v) \geq 2$.
Finding a characterization of globally 2-linked pairs (and an efficient algorithm for testing
global 2-linkedness)
has been one of the last remaining major open problems of
combinatorial rigidity theory in $\R^2$. 
We give a complete solution in this paper (Theorem \ref{thm:MAIN} below). Our result
gives an affirmative answer to \cite[Conjecture 5.9]{JJS}
and settles several other conjectures of this area.

Global rigidity and 
globally linked pairs can be defined in a similar manner in complex
frameworks $(G,p)$, where $p:V\to {\mathbb C}^d$ is a complex
realization of $G$. 
It is known that
generic global rigidity in $\R^d$ and ${\mathbb C}^d$ are
the same, see \cite{GT,JO}.
Jackson and Owen \cite{JO} proved that, unlike in the real setting,
global linkedness is a generic property
in ${\mathbb C}^d$.
They also characterized complex global linkedness in ${\cal R}_2$-connected
graphs (where the characterization is the same as in the real case, c.f. Theorem \ref{mconn})
and conjectured that global linkedness in $\R^2$ and ${\mathbb C}^2$
are the same, see \cite[Conjecture 5.4]{JO}. 

The following stress-based analogue of globally linked pairs was introduced 
by Garamv\"olgyi \cite{G}. We refer the reader to \cite{G} for the precise definition.
Roughly speaking, a pair $\{u,v\}$ of vertices of a graph $G$ is \emph{$d$-stress-linked in $G$} 
if for every generic framework $(G,p)$ in $\R^d$ and every framework $(G,q)$ 
for which every equilibrium stress of $(G,p)$ is also an equilibrium stress of $(G,q)$,
we have that every equilibrium stress of $(G+uv,p)$ is an equilibrium stress of $(G+uv,q)$.
In particular, adjacent pairs of vertices are $d$-stress linked in $G$. 
It was shown in \cite[Theorems 4.2 and 4.7]{G} that $d$-stress-linked pairs are globally $d$-linked, and
that a graph $G$ is globally $d$-rigid if and only
if $\{u,v\}$ is $d$-stress linked in $G$ for all pairs
$u,v\in V(G)$. 
Furthermore, a characterization of 2-stress linked pairs was also obtained \cite[Theorem 4.15]{G}.
A conjecture of the same paper \cite[Conjecture 6.1]{G}
suggests that
$d$-stress-linked and globally $d$-linked pairs are
the same in $\R^d$ for all $d\geq 1$.

Our main result is the following theorem, which
shows that the three notions mentioned above are indeed all the same. It
solves each of the corresponding conjectures
as well as three related conjectures, see \cite[Conjectures 5.12, 5.13]{JJS} and
\cite[Conjecture 3.13]{JJS2}.

\begin{theorem}
\label{thm:MAIN} 
Let $G=(V,E)$ be a graph and $u,v\in V$.  Then the following are equivalent:\\
(a) $\{u, v\}$ is globally linked in $G$ in $\R^2$,\\
(b) $\{u,v\}$ is globally  linked in $G$ in ${\mathbb C}^2$,\\
(c) $\{u,v\}$ is 2-stress-linked in $G$,\\
(d) either $uv \in E$ or there is an $\mathcal{R}_2$-connected subgraph $H = (V',E')$ of $G$ with $u,v \in V'$ and $\kappa_H(u,v) \geq 3$.
\end{theorem}

The essential new step, which is the main contribution of this paper is the proof of (a)$\to $(d).
The implications (d)$\to$(a) and (d)$\to$(b) were proved in \cite{JJS} and 
\cite{JO}, respectively, while (b)$\to$(a) follows from the definitions.
Furthermore, the equivalence of (c) and (d) 
was shown in \cite{G}.

We remark that the natural extension of (d)$\to$(a) 
fails in $d \geq 3$ dimensions, as 
there exist $(d+1)$-connected and $\cR_d$-connected graphs which are
not globally $d$-rigid, see \cite{JKT}.
However, it is still possible that the higher dimensional version of the
implication (a)$\to$(d)
holds in these dimensions, see \cite[Conjecture 5.2]{GJ}.

We use Theorem \ref{thm:MAIN} to deduce the characterization of the
globally 2-linked clusters in a graph $G$, which are the maximal vertex sets in which
all pairs of vertices are globally 2-linked in $G$, and point out a link
between these clusters
and a
conjectured characterization of global 3-rigidity.
The new methods developed in this paper can also be used to characterize globally
$d$-linked pairs for all $d\geq 1$ in special families of graphs.
We shall illustrate this by providing a characterization of 
the globally $d$-linked pairs in body-bar graphs. 
This result confirms a conjecture of Connelly, Jord\'an and Whiteley \cite{CJW}.


The structure of the paper is as follows.
Section \ref{sec:pre}
contains the basic definitions and results of rigidity theory
we shall need. Section \ref{sec:mco} contains old and new structural results 
concerning ${\cal R}_2$-connected graphs. Previous results on flexings of frameworks
as well as one of our key results, which gives a new method of using flexings to obtain
equivalent but non-congruent frameworks, are in Section \ref{sec:fle}.
In Section \ref{sec:mai} we prove Theorem \ref{thm:MAIN} by showing that
property (d) is indeed a necessary condition of
global 2-linkedness. We consider the globally 2-linked clusters and the
$d$-dimensional body-bar graphs in Sections \ref{sec:clu} and \ref{sec:bod}, respectively.
Concluding remarks, including a brief discussion of the algorithmic aspects, are in
Section \ref{sec:con}.

\section{Preliminaries}\label{sec:pre}

The rigidity matroid of a graph $G$ is a matroid defined on the edge set
of $G$ which reflects the rigidity properties of all generic realizations of
$G$.
For a general introduction to matroid theory we refer the reader to \cite{oxley}. For a detailed treatment of the $2$-dimensional rigidity matroid, see \cite{Jmemoirs}.

Let $(G,p)$ be a realization of a graph $G=(V,E)$ in $\RR^d$.
The \emph{ rigidity matrix} of the framework $(G,p)$
is the matrix $R(G,p)$ of size
$|E|\times d|V|$, where, for each edge $uv\in E$, in the row
corresponding to $uv$,
the entries in the $d$ columns corresponding to vertices $u$ and $v$ contain
the $d$ coordinates of
$(p(u)-p(v))$ and $(p(v)-p(u))$, respectively,
and the remaining entries
are zeros. 
The rigidity matrix of $(G,p)$ defines
the \emph{rigidity matroid}  of $(G,p)$ on the ground set $E$
by linear independence of rows of the
rigidity matrix. It is known that any pair of generic frameworks
$(G,p)$ and $(G,q)$ have the same rigidity matroid.
We call this the $d$-dimensional \emph{rigidity matroid}
${\cal R}_d(G)$
of the graph $G$, and denote its rank function 
%
%
by $r_d$. For a subgraph $H$ of $G$ we shall use $r_d(H)$ to mean $r_d(E(H))$. 
A graph $G=(V,E)$ is \emph{$\cR_d$-independent} if $r_d(G)=|E|$ and it is an \emph{$\cR_d$-circuit} if it is not $\cR_d$-independent but every proper 
subgraph $G'$ of $G$ is $\cR_d$-independent. 
An edge $e$ of $G$ is an \emph{$\cR_d$-bridge in $G$}
if  $r_d(G-e)=r_d(G)-1$ holds. Equivalently, $e$ is an $\cR_d$-bridge in $G$ if it is not contained in any subgraph of $G$ that is an $\cR_d$-circuit.
A pair $\{u,v\}$ of vertices is \emph{linked} in $G$ in $\R^d$ (or $d$-linked)  if $r_d(G+uv)=r_d(G)$ holds. By basic matroid theory, this is equivalent to the existence of an $\cR_d$-circuit in $G+uv$ containing the edge $uv$.

The following characterization of rigid graphs is due to Gluck.

\begin{theorem}\label{theorem:gluck}
\cite{Gluck}
\label{combrigid}
Let $G=(V,E)$ be a graph with $|V|\geq d+1$. Then $G$ is rigid in $\RR^d$
if and only if $r_d(G)=d|V|-\binom{d+1}{2}$.
\end{theorem}

We shall need three previous results concerning globally 2-linked pairs.
The first one characterizes globally 2-linked pairs in 
${\cal R}_2$-connected graphs.

\begin{theorem}\cite{JJS}
\label{mconn}
Let $G=(V,E)$ be an ${\cal R}_2$-connected graph and $x,y\in V$.
Then $\{x,y\}$ is globally 2-linked in $G$ if and only if
$\kappa_G(x,y)\geq 3$.
\end{theorem}

Let $H=(V,E)$ be a graph. The {\it 0-extension} operation adds a new vertex $z$ to $H$ as
well as two new edges incident with $z$. The {\it 1-extension} operation deletes an edge
$xy$, and adds a new vertex $z$ and three new edges incident with $z$, including $zx$ and $zy$.
The following two lemmas show that these operations preserve the property of
being "not globally 2-linked", at least  in certain cases.

\begin{lemma}\cite{JJS}
\label{lem0ext}
If $\{u,v\}$ is not globally 2-linked in $H$ and $G$ is a
0-extension of $H$, then $\{u,v\}$ is not
globally 2-linked in $G$.
\end{lemma}

\begin{theorem}\cite[Theorem 3.10]{JJS2}
\label{lem1ext}
Let $H=(V,E)$
 be a 2-rigid graph and let $G$ be obtained from $H$
 by a 1-extension on an $\cR_2$-bridge $uw\in E$.
 Suppose that 
 $\{x,y\}$ is not
 globally 2-linked in $H$ for some $x,y\in V$. Then $\{x,y\}$
 is not globally 2-linked in $G$.
\end{theorem}

\section{Structural properties of $\cR_2$-connected graphs}
\label{sec:mco}

Theorem \ref{theorem:2dimgloballyrigid} shows that global 2-rigidity and 
${\cal R}_2$-connectivity are closely related. The connectivity properties
of ${\cal R}_2(G)$ are also fundamental in the (proof of the) characterization of 
globally 2-linked pairs. This section is devoted to structural results
concerning ${\cal R}_2$-connected (sub)graphs.

Let ${\cal M}$ be a matroid on ground set $E$. 
We can define a relation on the pairs of elements of $E$ by
saying that $e,f\in E$ are
equivalent if $e=f$ or there is a circuit $C$ of ${\cal M}$
with $\{e,f\}\subseteq C$.
This defines an equivalence relation. The equivalence classes are 
the \emph{connected components} of ${\cal M}$.
Thus the connected components of ${\cal M}$ form a partition of $E$.
The matroid is \emph{connected} if there is only one equivalence class.
A graph $G$ is \emph{$\cR_d$-connected} if ${\cal R}_d(G)$ is connected.
The subgraphs of $G$ induced by the (edges of the) connected components of ${\cal R}_2(G)$
are called the {\it ${\cal R}_2$-components} of $G$.
An ${\cal R}_2$-component $H$ is {\it trivial} if $|E(H)|=1$, or equivalently,
if it corresponds to an ${\cal R}_2$-bridge of $G$. Otherwise it is {\it non-trivial}.
Some basic properties of the ${\cal R}_2$-components
are summarized in the next lemma.



\begin{lemma} \cite{JJconnrig,Jmemoirs}
\label{mconnbasic}
Let $G=(V,E)$ be a graph with ${\cal R}_2$-components $H_1,H_2,\dots,H_q$.
Then\\
(a) if $H_i$ is non-trivial, then it is a redundantly 2-rigid, 2-connected induced subgraph of $G$
with $|V(H_i)|\geq 4$, for $1\leq i\leq t,$\\
(b) $|V(H_i)\cap V(H_j)|\leq 1$ for $1\leq i<j\leq t$, and\\
(c) $r_2(G) = \sum_{i=1}^q r_2(H_i)$.
\end{lemma}

We shall also use the following properties.

\begin{lemma}
\label{newcomps}
Let $G=(V,E)$ be a graph with ${\cal R}_2$-components $H_1,H_2,\dots,H_q$,
and let $u,v\in V$ be a non-adjacent vertex pair with $u,v\in V(H_1)$.
Then the ${\cal R}_2$-components of $G+uv$ are $H_1+uv,H_2,\dots,H_q$.
In particular, the vertex sets of the ${\cal R}_2$-components of
$G$ and $G+uv$ are the same.
\end{lemma}

\begin{proof}
Since $H_1$ is 2-rigid by Lemma \ref{mconnbasic}(a), $H_1+uv$ is an ${\cal R}_2$-connected
subgraph of $G+uv$.
It suffices to show that there is no ${\cal R}_2$-circuit $C$ in $G+uv$
with $uv\in E(C)$ and $E(C)\cap E(H_i)\not= \emptyset$ for some
$2\leq i\leq q$. Suppose, for a contradiction, that such an ${\cal R}_2$-circuit 
exists and let $f\in E(C)\cap E(H_i)$.
Let $C'$ be an ${\cal R}_2$-circuit in $H_1+uv$ with $uv\in E(C')$.
Then the strong circuit axiom implies that there exists an ${\cal R}_2$-circuit
$C''$ in $G$ with $f\in E(C'')\subseteq (E(C)\cup E(C'))-uv$.
Since $E(C'')\cap E(H_1)\not= \emptyset$, it contradicts the
assumption that $H_1$ and $H_i$ are different
${\cal R}_2$-components of $G$.
\end{proof}

Let $H=(V,E)$ be an ${\cal R}_2$-connected graph with $|V|\geq 4$.
By Lemma \ref{mconnbasic}(a) $H$ is 2-connected. Let
$a,b\in V$. We say that the pair $\{a,b\}$ is a {\it 2-separator} of $H$
if $H-\{a,b\}$ is disconnected.
We say that two 2-separators $\{a,b\}$ and $\{a',b'\}$ of $H$ are {\it crossing},
if $a$ and $b$ are in different components of $H-\{a',b'\}$.
The next lemma is a corollary of \cite[Lemma 3.6]{JJconnrig}.

\begin{lemma} \cite{JJconnrig}
\label{nocross}
Suppose that $H$ is an ${\cal R}_2$-connected graph. Then
there are no crossing 2-separators in $H$.
\end{lemma}

Let $\{a,b\}$ be a 2-separator of $H$ and let $X$ be the union of the
vertex sets of some, but not all  components of $H-\{a,b\}$.
We say that the graphs $H_1=H[X\cup \{a,b\}]+ab$ and $H_2=(H-X)+ab$
are obtained from $H$ by {\it cleaving} $H$ {\it along the 2-separator} $\{a,b\}$.
Note that if $ab\in E$, then $H_1,H_2$ contain only one copy of $ab$.

\begin{lemma} \cite[Lemma 3.4]{JJconnrig}
\label{cleaving}
Suppose that $H_1,H_2$ are obtained from $H$ by cleaving along a 2-separator.
If $H$ is ${\cal R}_2$-connected, then $H_1$, $H_2$ are also ${\cal R}_2$-connected.
\end{lemma}


The {\it augmented graph} $\hat H$ is obtained from $H$ by adding an edge $ab$
for all 2-separators $\{a,b\}$ of $H$ with $ab\notin E$.
A maximal 3-connected subgraph of $\hat H$ is called a {\it 3-block}.
It was shown in \cite[Section 3]{JJconnrig} that each 3-block is 
%
${\cal R}_2$-connected, every edge $e=ab$ of $\hat H$ belongs to at least one 3-block,
and if $e$ belongs to two or more 3-blocks, then $\{a,b\}$ is a 2-separator.
%
The 3-blocks can be obtained
from $\hat H$ by recursively cleaving the graph along 
2-separators. Note that $r_2(H)=r_2(\hat H)$ and the 2-separators of $H$ and $\hat H$
are the same by Lemma \ref{mconnbasic}(a) and Lemma \ref{nocross}, respectively.


Let
$J_1,J_2,\dots,J_t$ be the 3-blocks of $H$. 
For each 2-separator $\{a,b\}$ of $H$ 
%
let $h_H(ab)$ denote the number of 3-blocks $J_i$, $1\leq i\leq t$, with
$\{a,b\}\subset V(J_i)$, and let
$k(H)=\sum (h_H(ab)-1)$, where the summation is over all 2-separators $\{a,b\}$ of $H$.
We say that an
ordering $(X_1,X_2,\dots,X_p)$ of $p$ subsets of $V$
is {\it $m$-shellable}, for some integer $m\geq 0$, if
$|(\cup_{i=1}^{j-1} X_i) \cap X_j|\leq m$ for all $2\leq j\leq p$.

\begin{lemma}
\label{hinge}
Let $H=(V,E)$ be an ${\cal R}_2$-connected graph with 
3-blocks
$J_1,J_2,\dots,J_t$.
Then\\
(a) $r_2(H)=\sum_{i=1}^t r_2(J_i) - k(H)$, \\
(b) $t= k(H)+1$, \\
(c) for every $e\in E(H)$ the vertex sets $V(J_i)$, $1\leq i\leq t$, have a 2-shellable
ordering such that $e$ is induced by the first set of the ordering.
\end{lemma}

\begin{proof}
The proof is by induction on $|V|$. For $|V|=4$ we have $H=K_4$ and the lemma trivially holds.
Suppose $|V|\geq 5$. 
If $H$ has no 2-separators, or equivalently, if $H$ is 3-connected, then
$t=1$, $k(H)=0$, and the
lemma is straightforward.
Let us assume that $H$ is not 3-connected and let 
$X\subset V$ be a minimal subset of vertices satisfying $|N_H(X)|=2$
and $V-X-N_H(X)\not= \emptyset$. Let $N_H(X)=\{a,b\}$.
By Lemma \ref{nocross} we have $N_H(X)=N_{\bar H}(X)$.
Let $H_1,H_2$ be the graphs obtained from $\bar H$
by cleaving along $\{a,b\}$, such that $V(H_1)=X\cup \{a,b\}$.
Note that for every designated edge we can choose $X$ so that $e\in E(H_2)$ holds.
By Lemma \ref{nocross} and the choice of $X$, 
$H_1$ is a 3-block of $H$ (say, $H_1=J_1$), and the
3-blocks of $H_2$ are $J_2,J_3,\dots, J_t$.
Furthermore, both $H_1$ and $H_2$ are ${\cal R}_2$-connected by Lemma \ref{cleaving}.
By induction, $r_2(H_2)=\sum_{i=2}^t r_2(J_i) - k(H_2),$ and $t-1=k(H_2)+1$.
Let $B_1$, $B_2$ be  ${\cal R}_2$-bases of $H_1,H_2$, respectively, with 
$ab\in E(B_1)\cap E(B_2)$. Then 
$B_1\cup B_2$ is an ${\cal R}_2$-base of $H$
by Lemma \ref{cleaving}. Hence $r_2(H)=r_2(H_1)+r_2(H_2)-1$.
Now we can deduce that (a) and (b) hold for $H$ by using that
$k(H)=k(H_2)+1$.
Moreover, the vertex sets of $J_2,J_3,...,J_t$ have a 2-shellable ordering such that
$e$ is induced by the first set of the ordering, by induction. By adding
$V(J_1)$ to the end of this ordering we obtain the ordering as required by (c).
\end{proof}

A 3-block of some non-trivial ${\cal R}_2$-component of a graph $G$
is said to be an {\it ${\cal R}_2$-block} of $G$.

\section{Equivalent realizations and flexings}
\label{sec:fle}

In this section we prove a key result, whose proof is based on
continuous motions of frameworks.

Let $G=(V,E)$ be a graph and let
$(G,p)$ be a $d$-dimensional framework. A {\it flexing} of the framework $(G,p)$ is
a continuous function $\phi:[0,1]\to {\R^{d|V|}}$ such that\\ 
(i) $\phi(0)=p$, \\
(ii) $(G,\phi(t))$ is equivalent to $(G,p)$ for all $t\in [0,1]$, and\\
(iii) $(G,\phi(t))$ is not congruent to $(G,p)$ for all $t\in (0,1]$.

The framework $(G,p)$ is said to be {\it flexible} if it has a flexing, and {\it rigid} otherwise.
Let us fix an ordering of the edges of $G$ and define
the {\it rigidity map} $f_G:{\mathbb R}^{d|V|}\to {\mathbb R}^{|E|}$ of $G$ by
$$f_G(p)= ( \dots, ||p(u)-p(v)||^2, \dots ),$$
where $uv\in E$, and $p(w)\in \R^d$ for $w\in V$.

For a smooth map $f:\R^n\to \R^m$, let $k = \max \{ \rank\ df(x) : x\in \R^n \}$, the maximum of the rank of
the Jacobian of $f$. We say that $x\in \R^n$ is a {\it regular point} of $f$ if $\rank \ df(x) = k$.
The image $f(p)$ is a {\it regular value} if each point in $f^{-1}(f(p))$ is a regular point.
Note that the Jacobian $df_G(p)$ of the rigidity map at some point $p\in \R^{d|V|}$ is given by $2R(G,p)$, where
$R(G,p)$ is the rigidity matrix of $(G,p)$.
The next lemma is a well-known "rigidity predictor".

\begin{lemma} \cite[Proposition 5.1]{R}
\label{roth}
Let $(G,p)$ be a $d$-dimensional framework. Suppose that $p$ is a regular
point of the rigidity map $f_G$ and the framework does not lie in a hyperplane of
$\R^d$. Then $(G,p)$ is rigid if and only if $\rank\ R(G,p)= d|V| - \binom{d+1}{2}$
and $(G,p)$ is flexible if and only if
$\rank\ R(G,p)< d|V| - \binom{d+1}{2}$.
\end{lemma}

We shall need the following statement concerning flexings 
that change the
distance between a designated vertex pair.

\begin{lemma}
\label{lem:flex}
Let $(G,p)$ be a generic realization of $G=(V,E)$ in $\R^d$ and
suppose that  $\{x,y\}$ is not $d$-linked in $G$ for some $x,y\in V$.
Let $(G,q)$ be an equivalent
realization with a $d$-dimensional affine span.
Then $(G,q)$ has a flexing $\phi:[0,1]\to \R^{d|V|}$ for which
$||\phi(t)(x)-\phi(t)(y)||\not= ||p(x)-p(y)||$ for all $t\in (0,1]$.
\end{lemma}

\begin{proof}
Let $H$ be a graph on vertex set $V$.
Since $p$ is generic and $f_H$ is a polynomial map between manifolds,
basic results of algebraic geometry imply that $f_H(p)$ is a regular value, 
see, e.g., \cite[Proposition 2.32]{GHT} or \cite[Theorem 9.6.2]{book}.
In particular,
$q$ is a regular point of $f_G$ and $\{x,y\}$ is not linked in $(G,q)$.
Thus
$\rank\ R(G,q) < d|V|-\binom{d+1}{2}$ and $\rank\ R(G+xy,q)=\rank\ R(G,q)+1$. Moreover,
since $(G,q)$ has a $d$-dimensional affine span, $G$ has a supergraph $\bar G$
for which $\rank\ R(\bar G,p)=d|V|-\binom{d+1}{2}-1$ and $\rank\ R(\bar G+xy,p)=d|V|-\binom{d+1}{2}$.
Hence $(G,q)$ has the desired flexing by Lemma \ref{roth}.
\end{proof}

In a graph $G=(V,E)$ the degree of a vertex $v\in V$ (resp. the set of neighbours of $v$)
is denoted by $\deg_G(v)$ (resp. $N_G(v)$). Hence we have $\deg_G(v)=|N_G(v)|$ if $G$ contains no
parallel edges.

The next theorem is the main result of this section.

\begin{theorem}
\label{lem:geo}
Let $G=(V,E)$ be a $d$-rigid graph, let $xy\in E$ be
an $\cR_d$-bridge in $G$
with $\deg_G(y)\geq d+2$,
and suppose that
the vertices in $N_G(y)-\{x\}$ are
pairwise globally $d$-linked in $G-y$.
Let $\{u,v\}$ be a pair of vertices with
$y\notin \{u,v\}$, such that
$\{u,v\}$ is not globally $d$-linked in $G-y$. 
Then $\{u,v\}$ is not globally $d$-linked in $G$.
\end{theorem}

\begin{proof}
Let $H=G-y$ and let $(H,p)$ be 
a generic $d$-dimensional realization of $H$
in which $\{u,v\}$ is not globally linked.
Then there exists an equivalent realization $(H,q)$
for which $|p(u)-p(v)|\not= |q(u)-q(v)|$.
Since the vertices in $N_G(y)-\{x\}$ are
pairwise globally $d$-linked in $H$, we may suppose, by applying an isometry, if necessary, that
$p(v)=q(v)$ for all $v\in N_G(y)-\{x\}$.

\begin{claim}
\label{clflex}
There exists
a vertex $w\in N_G(y)-\{x\}$ for which $\{w,x\}$ is not
$d$-linked in $H$.
\end{claim}

\begin{proof}
Suppose that $\{w,x\}$ is $d$-linked in $H$ for all $w\in N_G(y)-\{x\}$.
The fact that the vertices in $N_G(y)-\{x\}$ are pairwise globally $d$-linked in $H$
implies that they are also pairwise $d$-linked in $H$.
Thus the vertex set $N_G(y)$ forms a $d$-rigid cluster of size at least $d+1$ in $H$.
It follows that the edge $xy$ is induced by a $d$-rigid cluster in $G-xy$.
Hence $xy$ is $d$-linked in $G-xy$,
contradicting the assumption that $xy$ is an ${\cal R}_d$-bridge of $G$.
\end{proof}

The facts that $\deg_G(y)\geq d+2$ and
the subframework of $(H,q)$ on vertex set $N_G(y)-\{x\}$ is congruent to that
of $(H,p)$ shows that $(H,q)$ does not
lie
in a hyperplane of $\R^d$.
Thus we can apply Lemma \ref{lem:flex} to $(H,p)$ and $(H,q)$ to deduce that
$(H,q)$ has a flexing $\phi:[0,1]\to \R^{d|V|}$ such that for some $w\in N_G(y)-\{x\}$ we have
$||\phi(t)(x)-\phi(t)(w)||\not= ||q(x)-q(w)||$ for all $t\in (0,1]$.
We may assume, by continuity, and scaling the flexing, that 
$||p(u)-p(v)||\not= ||\phi(t)(u)-\phi(t)(v)||$ for all $t\in [0,1]$.
Since the vertices in $N_G(y)$ are pairwise $d$-linked in $H$, and $q$ is a regular point,
they are pairwise linked in $(H,q)$.
Hence we may assume that 
$\phi(t)(w)=q(w)$ for all $w\in N_G(y)-\{x\}$ and $t\in [0,1]$.
It follows that the flexing changes the position of $x$.
We shall compare $\phi(t)(x)$ to a "fixed" point $p(x)\in \R^d$, the position of $x$ in $(H,p)$.

For each $t\in [0,1]$ 
the points $a\in \R^d$ with
$||a-p(x)||=||a-\phi(t)(x)||$ lie on a hyperplane. As $\phi(t)(x)$ moves during the flexing, the union of
these hyperplanes contains 
an open ball $B\subset \R^d$. 
Let us choose a point $a^*\in B$ for which the multiset of the coordinates
of $p(v)$, $v\in V-\{y\}$, together with the coordinates of $a^*$, are
algebraically independent over the rationals.
Let $t'\in [0,1]$ such that 
$|a-p(x)|=|a-\phi(t')(x)|$. With these points in hand we are ready to define two
frameworks that verify the statement of the theorem.

Let $(G,p_1)$ be obtained from $(H,p)$ by adding vertex $y$
and putting $p_1(y)=a^*$. Let $(G,p_2)$ be obtained from
$(H,\phi(t'))$ by adding vertex $y$ and putting $p_2(y)=a^*$.
It follows from the choice of $a^*$ that $(G,p_1)$ is generic.
Since the subframeworks of $(G,p_1)$ and $(G,p_2)$ on vertex set
$N_G(y)-\{x\}$ are identical, the choice of $a^*$ implies that
$(G,p_1)$ and $(G,p_2)$  are equivalent.
As we have
$||p(u)-p(v)||\not= ||\phi(t')(u)-\phi(t')(v)||$,
we obtain that 
$\{u,v\}$ is not globally $d$-linked in $G$, as required.
\end{proof}

\section{Globally linked pairs in $\R^2$}
\label{sec:mai}

In this section we complete the proof of 
Theorem \ref{thm:MAIN}. As we discussed
in the Introduction, the theorem follows from
the next statement.

\begin{theorem}
\label{thm:main}
Let $G=(V,E)$ be a graph and let $u,v\in V$ be a pair of non-adjacent vertices.
If 
$\{u,v\}$ is globally 2-linked in $G$, then
there is an $\cR_2$-component $H$ of $G$ with
$u,v\in V(H)$ and $\kappa_H(u,v)\geq 3$.
\end{theorem}

\begin{proof}
Let $\{u,v\}$ be a non-adjacent vertex pair in $G$.
We shall prove,
by
induction on $|V|$,
that if
\begin{equation}
\label{no}
\hbox{there is no}\ {\cal R}_2\hbox{-component}\ H \ \hbox{with}
\ u,v\in V(H)\ \hbox{and}\ \kappa_H(u,v)\geq 3,
\end{equation}
in $G$, then $\{u,v\}$ is not globally 2-linked in $G$.
The cases $|V|\leq 3$ are trivial, so we may
assume that $|V|\geq 4$.
%
First we show that we can add new edges to $G$ so that
the resulting graph is 2-rigid, each of its ${\cal R}_2$-blocks is a complete
subgraph, $u$ and $v$
remain non-adjacent, and (\ref{no}) is preserved.
To prove this first observe that, since 
$K_{|V|}$ minus an edge is 2-rigid for $|V|\geq 4$, there exists a set $B$ of new edges,
$uv\notin B$, such that $G+B$ is 2-rigid and each edge of $B$ is an
${\cal R}_2$-bridge in $G+B$.
Thus the addition of $B$ makes the graph 2-rigid and preserves (\ref{no}).
Next observe that 
adding an edge that connects a pair of non-adjacent vertices
of an $\cR_2$-block does not change the vertex sets of the
${\cal R}_2$-blocks and also preserves (\ref{no}). 
Hence we can make all the 
${\cal R}_2$-blocks complete subgraphs. 
Since
it suffices to show that
$\{u,v\}$ is not globally 2-linked in a supergraph of $G$, in the rest of the proof we may assume
that $G$ is 2-rigid and each ${\cal R}_2$-block is a complete
subgraph of $G$.

We say that a vertex $y\in V$ is {\it reducible} if either 
$y$ belongs to a unique $\cR_2$-block of $G$ and $y$ is incident with
at most one $\cR_2$-bridge, or $y$ belongs to no $\cR_2$-block of $G$ and
$y$ is incident with at most three $\cR_2$-bridges of $G$.

\begin{claim}
\label{claim}
$G$ has at least three reducible vertices.
\end{claim}

\begin{proof}
Let $H_1,H_2,...,H_q$ be the non-trivial $\cR_2$-components and
let $J_1,J_2,...,J_t$ be the $\cR_2$-blocks of $G$. Let
$n_i=|V(J_i)|$ for $1\leq i\leq t$.
We have $n_i\geq 4$ for all $1\leq i\leq t$.
By Lemma \ref{hinge} we have
\begin{equation}
\label{hingerank}
r_2(H_i)=\sum_{J_i\subseteq H_i} (2n_i-3)-k(H_i),
\end{equation}
for $1\leq i\leq q$.

Let $G'$ be the subgraph of $G$ induced by the
non-trivial $\cR_2$-components and 
let $F\subseteq E$ be the set of $\cR_2$-bridges in $G$.
For each $\cR_2$-block $J_i$ let $X_i$ be the set, and $x_i$ be the number of  vertices in $J_i$ that belong to no other ${\cal R}_2$-block, let $Y_i=V(J_i)-X_i$,
and $y_i=|Y_i|$. Then
we have $n_i=x_i+y_i$, $1\leq i\leq t$.
Let $X=\cup_{i=1}^t X_i$, $Y=\cup_{i=1}^t Y_i$, and let
$Z=V-(X\cup Y)$ be the set of vertices that belong to no non-trivial $\cR_2$-component in $G$.
Note that $\sum_{i=1}^t y_i\geq 2|Y|$, since each vertex of $Y$
contributes to the left hand side by at least two.

We shall prove the claim by a counting argument, focusing on
the vertex set $X\cup Z$ in the subgraph $G_F$, where $G_F=(V,F)$.
For each $x\in X$ let $c(x)=\max \{2-d_{G_F}(x), 0\}$ and for each
$z\in Z$ let $c(z)=\max \{4-d_{G_F}(z), 0\}$.
Since $G$ is 2-rigid with $|V|\geq 4$, we have
$c(v)\leq 2$ for all $v\in X\cup Z$.
By counting degrees in $G_F$ we obtain
\begin{equation}
\label{count}
|F|\geq \frac{\sum_{i=1}^t 2x_i + 4|Z|-\sum_{v\in X\cup Z} c(v)}{2} = \sum_{i=1}^t x_i + 2|Z| - \frac{\sum_{v\in X\cup Z} c(v)}{2}.
\end{equation}

\noindent
Let us define $C=\frac{\sum_{v\in X\cup Z} c(v)}{2}$ for simplicity.
Then
\begin{align*}
    2|V|-3&=r_2(G) \qquad (\text{by the 2-rigidity of $G$ }) \\
    &=\sum_{i=1}^q r_2(H_i) + |F|\qquad (\text{by Lemma~\ref{mconnbasic}}(c)) \\
    &\geq \sum_{i=1}^t (2n_i-3) - \sum_{i=1}^q k(H_i) + \sum_{i=1}^t x_i +2|Z| 
- C \qquad (\text{by (\ref{hingerank}) and (\ref{count})}) \\
&=2\sum_{i=1}^t x_i + \sum_{i=1}^t y_i + \sum_{i=1}^t (x_i+y_i) - 3t - \sum_{i=1}^q k(H_i) + 2|Z| - C  \qquad \text{(by $n_i=x_i+y_i$)}\\
&\geq 2|V|+4t - 3t - \sum_{i=1}^q k(H_i) - C \quad \text{(by $x_i+y_i\geq 4$ and $\sum_{i=1}^t 2x_i+\sum_{i=1}^t y_i+2|Z|\geq 2|V|$)}\\
&=2|V| + q - C \qquad \text{(by $t+\sum_{i=1}^q k(H_i)=q$ by Lemma~\ref{hinge}(b)).}
\end{align*}
As $q\geq 0$, we must have $C \geq 3$.
Since $c(v)\leq 2$ for each $v\in X\cup Z$, this implies that
there exist at least three vertices $v$ in $X\cup Z$
with $c(v)\geq 1$. The claim follows, as
these vertices are all reducible.
\end{proof}

By Claim \ref{claim} there exists a reducible vertex $y\in V$
with $y\notin \{u,v\}$.
First suppose that 
$y$ belongs to a unique $\cR_2$-block $J$ of $G$ and $v$ is incident with
no $\cR_2$-bridges. Then $G[N_G(y)]$ is a complete graph on at least three
vertices. Furthermore, 
$G-y$ satisfies (\ref{no}).
We can now deduce
by induction, that $\{u,v\}$ is not globally 2-linked in $G-y$.
Since the neighbour set of $y$ in $G$ is complete, this implies
that $\{u,v\}$ is not globally 2-linked in $G$.

Next suppose that 
$y$ belongs to a unique $\cR_2$-block $J$ of $G$ and $y$ is incident with
one $\cR_2$-bridge, say $xy$. Then $G[N_G(y)-\{x\}]$ is a complete graph on at least three
vertices. Thus
Lemma \ref{lem:geo} implies that
$\{u,v\}$ is not globally 2-linked in $G$.

Finally, suppose that  
$y$ belongs to no $\cR_2$-block of $G$ and
$y$ is incident with at most three $\cR_2$-bridges of $G$.
Since $G$ is 2-rigid, we have $2\leq d_G(v)\leq 3$.
If $\deg_G(y)=2$, then Lemma \ref{lem0ext} implies that 
$\{u,v\}$ is not globally 2-linked in $G$, so we may assume that
$\deg_G(v)=3$.

\begin{claim}
\label{claim2}
Let $N_G(y)=\{x_1,x_2,x_3\}$. Then 
$\{x_i,x_j\}$ is not 2-linked in $G-y$ for some 
$1\leq i<j\leq 3$.
\end{claim}

\begin{proof}
Suppose that each pair of vertices in $N_G(y)$ is 2-linked
in $G-y$. Then the well-known fact that the 0-extension operation preserves
2-rigidity implies that 
$G-yx_i$ is 2-rigid for $1\leq i\leq 3$. This 
contradicts the assumption that the edges incident with
$y$ are $\cR_2$-bridges in $G$.
\end{proof}

By Claim \ref{claim2}
we may suppose, by relabelling the neighbours of $y$, if necessary, that
$\{x_1,x_2\}$ is not 2-linked in $G-y$.
Then
$G'=G-y+x_1x_2$ is 2-rigid, and
$x_1x_2$ is an $\cR_2$-bridge in $G'$.
Since $G$ can be obtained from $G'$ by a 1-extension on edge $x_1x_2$,
Lemma \ref{lem1ext} implies that $\{u,v\}$ is not globally 2-linked in $G$.
This completes the proof of the theorem.
\end{proof}

Theorem \ref{thm:main} implies the 
following statement, which is a weaker version of the
2-dimensional case
of Theorem \ref{theorem:hendrickson}.

\begin{corollary}
Let $G=(V,E)$ be a graph with $|V|\geq 4$ and suppose that
every generic realization of $G$ in $\R^2$ is globally rigid.
Then $G$ is redundantly 2-rigid.
\end{corollary}

\begin{proof}
We may assume that $G$ is 2-rigid. We shall prove that
$G$ has no ${\cal R}_2$-bridges.
For a contradiction suppose that
$G-xy$ is not 2-rigid for some $xy\in E$. 
Since $|V|\geq 4$, there exists a pair $\{u,v\}(\not= \{x,y\})$ such that
$u$ and $v$ do not belong to the same 2-rigid subgraph of $G-xy$.
The fact that ${\cal R}_2$-connected graphs are 2-rigid implies that there is no ${\cal R}_2$-connected subgraph of $G-xy$
which contains both $u$ and $v$. 
As $xy$ is an ${\cal R}_2$-bridge in $G$, the same holds in $G$, too.
By Theorem \ref{thm:main} we obtain that $\{u,v\}$ is not globally 2-linked in $G$, which means that there exists a generic realization $(G,p)$ of $G$ in which
$\{u,v\}$ is not globally 2-linked.
Thus $(G,p)$ is not globally rigid in $\R^2$, a contradiction.
\end{proof}

By using that global rigidity is a generic property, we
obtain the original statement of Theorem \ref{theorem:hendrickson}
in the $d=2$ case. 

\section{The globally linked clusters in $\R^2$}
\label{sec:clu}

The 2-dimensional {\it globally linked closure}, denoted by {\em $\glc_2(G)$}, is the graph obtained from $G$
by adding an edge $uv$ for all pairs $\{u,v\}$ of non-adjacent globally 2-linked vertices of $G$. The
{\it globally 2-linked clusters} of $G$ are the vertex sets of the
maximal complete subgraphs in $\glc_2(G)$.

Theorem \ref{thm:MAIN} and Lemma \ref{newcomps} give the following characterization.

\begin{lemma}
\label{glclusters}
Let $G=(V,E)$ be a graph.
Then
the globally 2-linked clusters of $G$ of size at least four are precisely
the vertex sets of the $\cR_2$-blocks
of $G$.
Furthermore, 
%
an edge $e\in E(\glc_2(G))$ is not induced by a globally 2-linked cluster of
size at least four if and only if $e$ is an ${\cal R}_2$-bridge of $G$.
%
%
\end{lemma}

The main result of this section, Theorem \ref{cover} below, is motivated by a conjectured characterization of
global 3-rigidity, see \cite[Conjecture 4.9]{survey}.
The truth of this conjecture would imply that if a graph $G$ (which is not a copy of $K_{5,5}$)
is not globally 3-rigid, then 
this fact can be certified by a set $F\subseteq E(G)$ and a family of subsets of $V(G)$ of size at least five
which forms a 4-shellable "non-trivial tight cover" of $G-F$, see \cite{survey}. 
The theorem implies that there is a similar certificate of being not globally 2-rigid,
and it can be obtained from the globally 2-linked clusters of the graph. It is conceivable that
a similar phenomenon holds in $\R^3$.

We shall need one more lemma on ${\cal R}_2$-components in the proof.
For a graph $G=(V,E)$ and its subgraph $H$ we call $(N_G(V-V(H))\cap V(H)$
the {\it vertices of attachment} of $H$.

\begin{lemma}
\label{3comps}
Let $G=(V,E)$ be a graph
with at least three ${\cal R}_2$-components. Then
$G$ has at least three ${\cal R}_2$-components
with at most two vertices of attachment.
%
%
\end{lemma}

\begin{proof}
The proof is by induction on $|V|$.
For $|V|=3$ we have $G=K_3$, in which case the statement is clear.
Suppose $|V|\geq 4$.
If $G$ is disconnected, or has a cut-vertex, then the lemma follows easily by induction,
using Lemma \ref{mconnbasic}.
So we may assume
that $G$ is $2$-connected.
Let $H_1,H_2,\dots,H_q$ be the ${\cal R}_2$-components of $G$,
let $n_i=|V(H_i)|$, and let $y_i$ be the number of attachment vertices
of $H_i$. Let $x_i=n_i-y_i$.
Since $G$ is $2$-connected, we have $y_i\geq 2$ for all $1\leq i\leq q$.
For a contradiction 
suppose that $n_i\geq y_i\geq 3$ for all but at most
two ${\cal R}_2$-components of $G$.
Since each $H_i$ is 2-rigid, we can use Lemma \ref{mconnbasic}(c) to deduce that
$2|V|-3\geq r_2(G)=\sum_{i=1}^q r_2(H_i) = \sum_{i=1}^q (2n_i-3) = 2\sum_{i=1}^q n_i - 3q=$
$=(2\sum_{i=1}^q x_i + \sum_{i=1}^q y_i) + \sum_{i=1}^q y_i - 3q \geq 2|V| + 3q - 2 - 3q \geq 2|V|-2,$
a contradiction.
\end{proof}

Let $G=(V,E)$ be a graph and let ${\cal X}$ be a family of subsets of $V$.
For a pair $u,v\in V$ let $h_{\cal X}(uv)$ denote the number of sets $X\in \cal X$
with $u,v\in X$. Let $H({\cal X})=\{ \{u,v\}: h_{{\cal X}}(uv)\geq 2\}$.

\begin{theorem}\label{cover}
Let $G=(V,E)$ be a graph, let ${\cal C}=\{C_1,C_2,...,C_s\}$ be the  
globally 2-linked clusters of $G$ of size at least four, and let $F\subseteq E$ be the set of edges of $G$ not induced by the members of ${\cal C}$.  
Then 
\begin{equation}
\label{tight}
r_2(G) = |F| + \sum_{i=1}^s (2|C_i|-3) - \sum_{\{u,v\}\in H({\cal C})} (h_{\cal C}(uv)-1).
\end{equation}
Furthermore, ${\cal C}$ has a 3-shellable ordering.
%
\end{theorem}

\begin{proof}
We may assume that $G={\rm glc}_2(G)$.
By Lemma \ref{glclusters} $F$ is the set of ${\cal R}_2$-bridges of $G$.
Let $H_1,H_2,...,H_q$ be the non-trivial ${\cal R}_2$-components of $G$.
By Lemma \ref{mconnbasic}(b) $H({\cal C})$ consists of the vertex pairs of the
2-separators $\{a,b\}$ of the ${\cal R}_2$-components, and $h_{\cal C}(ab)=h_{H_i}(ab)$, where
$a,b\in V(H_i)$ and $h_{H_i}(ab)$ denotes the number of ${\cal R}_2$-blocks in $H_i$
that contain both $a$ and $b$ (as defined in  Section \ref{sec:mco}). We can now use Lemma \ref{mconnbasic}(c), Lemma \ref{hinge}(a), and the fact that
each ${\cal R}_2$-block is 2-rigid to deduce that
$$r_2(G)= |F| + \sum_{i=1}^q r_2(H_i)= |F|+\sum_{i=1}^q \ (\sum \{r_2(J) : J\ \hbox{is an}\ {\cal R}_2\hbox{-block of}\ H_i\}) - k(H_i))=$$
$$= |F| + \sum_{i=1}^s (2|C_i|-3) - \sum_{u,v\in H({\cal C})} (h_{\cal C}(uv)-1),$$
which proves the first part of the statement.

We prove the second part by induction on $q$. For $q=1$ we are done by
Lemma \ref{hinge}(c).
Suppose that $q\geq 2$. By Lemma \ref{3comps}, applied to $G-F$, there is a non-trivial ${\cal R}_2$-component, say $H_q$, of $G$ with at most
two vertices of attachment.
Let $G'$ be obtained from $G$ deleting the non-attachment vertices and the edges of $H_q$.
The non-trivial ${\cal R}_2$-components of $G'$ are $H_1,H_2,...,H_{q-1}$. 
By induction, 
the globally 2-linked clusters of $G'$
have a 3-shellable ordering.
We can extend this 
to
a 3-shellable ordering of $\cal C$ by adding a 2-shellable ordering of
the vertex sets of the ${\cal R}_2$-blocks of $H_q$ in such a way that if there
exists an attachment vertex $x\in V(H_q)$, then we choose such an ordering in which the first ${\cal R}_2$-block
contains 
an edge of $H_q$ incident with $x$. Such an ordering exists by Lemma \ref{hinge}(c).
Since $H_q$ has at most two vertices of attachment, its first
${\cal R}_2$-block has at most two vertices in common with the preceeding sets.
Furthermore, the choice of $x$ and the 2-shellability within $H_q$ imply that 
the extended ordering is 3-shellable.
\end{proof}

In Theorem \ref{cover}
we cannot replace 3-shellable by 2-shellable, see
Figure~\ref{fig:example}.

\begin{figure}[t]
\centering
\includegraphics[scale=0.7]{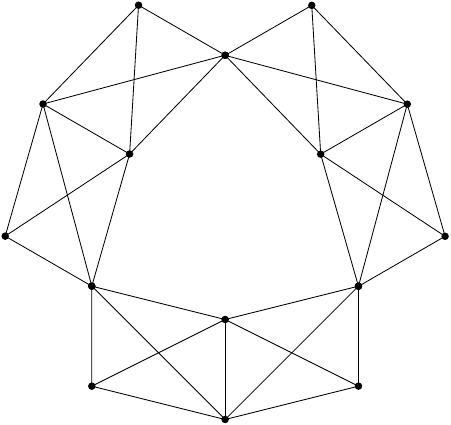}
\caption{The globally 2-linked clusters in this graph are the vertex sets of the
six copies of $K_4$. No ordering of these sets is 2-shellable.}
\label{fig:example}
\end{figure}

\section{Globally linked pairs in $d$-dimensional body-bar graphs}
\label{sec:bod}

Let $H=(V,E)$ be a loopless multigraph.
The {\it body–bar graph induced
by $H$}, denoted by $G_H$, is the graph obtained from $H$ by replacing each vertex $w\in V$ by a complete
graph $B_w$ (the ‘body’ of $w$) on $\deg_H(w)$ vertices and replacing each edge $wz$ by an edge (a ‘bar’)
between $B_w$ and $B_z$ in such a way that the bars are pairwise disjoint.
The $d$-rigidity and the global $d$-rigidity of body-bar graphs have been characterized in terms of
the "tree-connectivity"
of the underlying 
multigraph $H=(V,E)$.
For a partition $\cal P$ of $V$ let 
$e_H({\cal P})$ denote the
number of edges of $H$ that connect distinct parts of $\cal P$. We say that $H$ is 
{\it $k$-tree-connected}, for some integer $k\geq 1$, if
\begin{equation}
\label{htc}
e_H({\cal P}) \geq  k(t-1)
\end{equation}
for all partitions ${\cal P}$ of $V$ into $t\geq 1$ parts.
We call $H$
{\it highly $k$-tree-connected} 
%
if (\ref{htc}) holds with strict inequality 
whenever $t\geq 2$. Note that a single vertex, $K_1$, is highly $k$-tree connected for
all $k\geq 1$.
The following theorem is due to Tay.

\begin{theorem} \cite{tay}
\label{thm:tay}
Let $H=(V,E)$ be a multigraph with $|V|\geq  2$ and $|E|\geq  2$ and let $G_H$ be the body–bar graph
induced by $H$. Let $d\geq 1$ be an integer. Then 
$G_H$ is $d$-rigid 
if and only if 
$H$ is $\binom{d+1}{2}$-tree-connected.
\end{theorem}

Globally $d$-rigidity for body-bar graphs turns out to be the same as redundant $d$-rigidity, in the following sense. 

\begin{theorem} \cite{CJW}
\label{thm:cjw}
Let $H=(V,E)$ be a multigraph with $|V|\geq  2$ and $|E|\geq  2$ and let $G_H$ be the body–bar graph
induced by $H$. Let $d\geq 1$ be an integer. Then 
$G_H$ is globally $d$-rigid 
if and only if 
$H$ is highly $\binom{d+1}{2}$-tree-connected.
\end{theorem}


It was conjectured in \cite{CJW} that
a non-adjacent pair $\{u,v\}$ is globally $d$-linked in $G_H$
if and only if there is a globally $d$-rigid subgraph of $G_H$ that contains both $u$ and $v$.
In the proof of this conjecture we need some further notions and  
structural results.

Let $k\geq 1$ be an integer.
A maximal $k$-tree-connected 
subgraph of a
multigraph $H=(V,E)$ is called a 
{\it $k$-superbrick} of $H$.
It was shown in \cite{JJbrick} that the vertex sets of the 
superbricks of $H$ form a partition of $V$. 
Let ${\cal M}_k(H)$ be the matroid
union of $k$ copies of the cycle matroid of $H$.
The $k$-superbricks correspond to the non-trivial connected components of this matroid:

\begin{lemma} \cite[Lemma 2.11]{JJbrick}
\label{jjbridge}
Let $H=(V,E)$ be a multigraph and let $k\geq 1$ be an integer. 
Let $F\subseteq E$ be the set of bridges in ${\cal M}_k(H)$.
Then the $k$-superbricks of $G$ are the connected components of
the graph $(V,E-F)$.
\end{lemma}

We are ready to prove the 
main result of this section
and confirm the above mentioned
conjecture of Connelly, Jord\'an, and Whiteley
\cite[Section 6.2]{CJW} on globally $d$-linked pairs, which has been unsolved 
even for $d=2$.
For a multigraph $H=(V,E)$ and $X\subseteq V$ let 
$B_X=\cup_{w\in X}V(B_w)$ be the union of the vertex sets of the corresponding bodies in $G_H$.

\begin{theorem}
\label{thm:mainbbh}
Let $H=(V,E)$ be a multigraph with $|V|\geq 2$ and $|E|\geq 2$ and let
$G_H$ be the body-bar graph induced by $H$. Let $u,v\in V(G_H)$ be a non-adjacent pair and
let $d\geq 1$.
Then $\{u,v\}$ is globally $d$-linked in $G_H$
if and only if 
there exists a $\binom{d+1}{2}$-superbrick $S$ of $H$ with $u,v\in B_{V(S)}$.
%
\end{theorem}

\begin{proof}
%
Let $S$ be a $\binom{d+1}{2}$-superbrick of $H$ with $u,v\in B_{V(S)}$.
Since $u$ and $v$ are non-adjacent, we have $|V(S)|\geq 2$.
Theorem \ref{thm:cjw} implies that $G_{S}$, which is a body-bar subgraph of $G_H$,
is globally $d$-rigid.
This subgraph may properly intersect some bodies in $G_H$. However, by using
that
$S$ has minimum degree at least $\binom{d+1}{2}+1\geq d+1$, it follows that
these intersections have cardinality at least $d+1$. 
Thus $G_H[B_{V(S)}]$ is also globally $d$-rigid.
This proves sufficiency, and shows that the
$\binom{d+1}{2}$-superbricks of $H$ induce a partition of $G_H$
into globally $d$-rigid subgraphs.

Let us consider necessity. Our goal is to show that if the vertices of a non-adjacent pair $\{u,v\}$ of $G_H$ belong to
different members of this partition,
then they are not globally $d$-linked.
By adding edges to $H$ without introducing new $\binom{d+1}{2}$-superbricks (and hence adding new vertices and edges to $G_H$)
we may assume that $H$ is $\binom{d+1}{2}$-tree-connected, and hence $G_H$ is $d$-rigid
by Theorem \ref{thm:tay}. 
Let $C_1,C_2,\dots, C_q$ be the ${d+1\choose 2}$-superbricks
of $H$,
and let $B_i=B_{V(C_i)}$ for $1\leq i\leq q$.
Then
${\cal B}=\{B_1,B_2,...,B_q\}$ is a partition of $V(G_H)$,
and by our assumption, $u$ and $v$ belong to different
members of ${\cal B}$.

Since $G_H$ is $d$-rigid
and $uv\notin E(G_H)$, $G_H+uv$ contains an ${\cal R}_d$-circuit $J$ with $uv\in E(J)$.
Since $J$ is 2-edge-connected, there is an edge $xy\in E(J)$,
different from $uv$,
such that $x$ and $y$ belong to different 
members of ${\cal B}$.
By symmetry we may assume that $y\notin \{u,v\}$.
Hence, by the construction of the body-bar graph, $H$ has a 
unique edge $f\in E$ corresponding to $xy$,
and $f\notin E(C_i)$ for $1\leq i\leq q$.
Therefore, $f$ is a bridge in ${\cal M}_{d+1 \choose 2}(H)$
by Lemma \ref{jjbridge}, hence
$xy$ is an ${\cal R}_d$-bridge in $G_H$ by Theorem \ref{thm:tay}. It follows that $G_H-xy$ is not $d$-rigid.

Furthermore, the existence of the ${\cal R}_d$-circuit
$J$ with $xy,uv\in E(J)$ implies that
$G_H-xy+uv$ is $d$-rigid. By using that $G_H-xy$ is not $d$-rigid, we obtain that $\{u,v\}$ is not $d$-linked in $G_H-xy$.
Thus $\{u,v\}$ is not globally $d$-linked in $G_H-y$.

The body-bar structure and the $d$-rigidity of $G_H$, 
together with Theorem \ref{thm:tay}, imply that
$\deg_{G_H}(y)\geq d+2$ and 
the vertices in
$N_{G_H}(y)-\{x\}$ induce a complete subgraph in $G_H$.
%
So the conditions of Theorem \ref{lem:geo} are satisfied, and
we obtain that $\{u,v\}$ is not globally $d$-linked in $G_H$.
This completes the proof.
\end{proof}


We have the following corollary, which gives affirmative answers to 
two conjectures from \cite{G} and
\cite{GJ}, respectively, mentioned in the Introduction, 
in the special case of body-bar graphs. 

\begin{corollary}
\label{coro}
Let $H=(V,E)$ be a multigraph with $|V|\geq 2$ and $|E|\geq 2$, 
and let $G_H$ be the body-bar graph induced by $H$.
Let $u,v\in V(G_H)$ and let $d\geq 1$.
Then\\
(a) $X\subseteq V(G_H)$ is a globally $d$-linked cluster of $G_H$
if and only if 
$X=B_{V(S)}$ for some 
$\binom{d+1}{2}$-superbrick $S$ of $H$, or $X$ is the
end-vertex pair of an edge of $G_H$ not induced by such a cluster,\\
(b) $\{u,v\}$ is globally $d$-linked in $G_H$ if and only if 
either $uv\in E(G_H)$ or there is
an ${\cal R}_d$-connected subgraph $G'$ of $G_H$ with $\kappa_{G'}(u,v)\geq d+1$, and\\
(c) $\{u,v\}$ is $d$-stress-linked in $G_H$ if and only if
$\{u,v\}$ is globally $d$-linked in $G_H$.
\end{corollary}

\begin{proof} (a) follows directly from Theorem \ref{thm:mainbbh}.

(b) Let us consider a non-adjacent globally $d$-linked pair $\{u,v\}$.
It belongs to a non-trivial globally $d$-linked cluster of $G_H$. By (a) 
this cluster induces 
a globally $d$-rigid subgraph $G'$ of $G_H$. Hence $G'$ is
$(d+1)$-connected. By a result of 
\cite{GGJ} $G'$ is ${\cal R}_d$-connected. Thus (b) holds.

(c) Necessity was proved in \cite[Theorem 4.2]{G}. Suppose that $\{u,v\}$ is a non-adjacent
globally $d$-linked pair in $G_H$. 
It follows from (the proof of) (b) that there is a globally $d$-rigid subgraph $G'$ of $G_H$
with $u,v\in V(G')$. The result in \cite[Proposition 4.3]{G} implies that a graph is globally $d$-rigid
if and only if each pair of its vertices is $d$-stress-linked.
Thus $\{u,v\}$ is $d$-stress-linked in $G'$. By \cite[Lemma 4.9]{G} it is also
$d$-stress-linked in $G_H$.
%
\end{proof}

Corollary \ref{coro}(a)
implies that
the non-trivial globally $d$-linked clusters of $G_H$ are determined 
by an appropriate partition of the vertex set of $H$.
We conjecture that this property holds in body-hinge
graphs, too. The global $d$-rigidity of these graphs, 
which can be described as collections of rigid bodies in which 
some pairs of bodies share $d-1$ vertices, has been
characterized in \cite{JKT}.
Given a multigraph $H$ and $k\geq 1$, the graph $kH$
is obtained from $H$ by replacing every edge $e$ by $k$
copies of $e$.

\begin{conjecture}
Let $G_H$ be the body-hinge graph induced by multigraph $H$
and let $d\geq 3$. Then
a pair $\{u,v\}$ is globally $d$-linked in 
$G_H$ if and only if there is a 
$\binom{d+1}{2}$-superbrick $S$ of $(\binom{d+1}{2}-1)H$
which contains the vertices of the bodies of $u$ and $v$.
\end{conjecture}

\section{Concluding remarks}
\label{sec:con}

\subsection{Algorithms}

The combinatorial characterizations given in
Theorems \ref{thm:MAIN} and \ref{thm:mainbbh} imply
that the globally 2-linked pairs in a graph $G$ and the
globally $d$-linked pairs in a body-bar graph $G_H$ can be found in
polynomial time. These corollaries follow from the fact that
the ${\cal R}_2$-components, the  pairs $\{u,v\}$ of vertices with $\kappa_G(u,v)\geq 3$,
and the $k$-superbricks of a multigraph can be found efficiently, see \cite{JJbrick,Jmemoirs}.

Another algorithmic implication is concerned with the family of
$d$-joined graphs, defined in \cite{GJ}.
A graph $G=(V,E)$ is said to be 
{\it $d$-joined}, for some $d\geq 1$, if
$G$ is $d$-rigid, and for all $u,v\in V$ the pair $\{u,v\}$
is globally $d$-linked in $G$ if and only if $uv\in E$ or 
$\kappa_G(u,v)\geq d+1$. For example, ${\cal R}_2$-connected graphs are
$2$-joined by Theorem \ref{mconn}. It was pointed out in \cite{GJ} that $d$-joined
graphs have various interesting properties, but it remained an open
problem to develop an algorithm for testing whether a graph is $d$-joined, for $d\geq 2$. Theorem \ref{thm:MAIN} gives rise to such an algorithm for $d=2$. It also
provides an affirmative answer to the ($2$-dimensional version of) \cite[Conjecture 5.7]{GJ},
which has further algorithmic consequences, see \cite{GJ}.

\subsection{Uniquely localizable vertices}

The theory of globally rigid graphs and globally linked pairs 
has several applications, for example, in 
localization
problems of wireless sensor networks, see, e.g., \cite{JJchapter}.
The following definition in \cite{JJS} was motivated by this context.

Let us assume $d=2$ and let
$(G,p)$ be a generic framework with a designated set $P\subseteq V(G)$
of vertices. We say that a vertex $v\in V(G)$ is {\it uniquely localizable}
with respect to $P$
if whenever
$(G,q)$ is equivalent to $(G,p)$ and
$p(b)=q(b)$ for all vertices $b\in P$, then we also have $p(v)=q(v)$.
We call a vertex $v$ {\it uniquely localizable} in graph $G$
with respect to $P\subseteq V(G)$ if $v$ is uniquely localizable with respect to $P$ in all generic frameworks $(G,p)$.
Let $G+K(P)$ denote the graph obtained from $G$ by adding all edges $bb'$
for which $bb'\notin E(G)$ and $b,b'\in P$.

Theorem \ref{thm:MAIN} implies the following characterization of uniquely localizable vertices, confirming \cite[Conjecture 6.3]{JJS}.

\begin{theorem}
    Let $G=(V,E)$ be a graph, $P\subseteq V$, $v\in V-P$. Then $v$ is uniquely localizable in $G$ with respect to $P$ in $\R^2$ if and only if $|P|\geq 3$ and
    there is an ${\cal R}_2$-component $H$ in $G+K(P)$ with $P+v\subseteq V(H)$
    and $\kappa_H(v,b)\geq 3$ for all $b\in P$.
\end{theorem}

\section{Acknowledgements}

TJ was 
supported by  RIMS (Research Institute for Mathematical Sciences), Kyoto University,
the National Research, Development and Innovation Office
of Hungary, grant no. Advanced 152786, 
and the MTA-ELTE Momentum Matroid Optimization Research Group. 
This work was supported by the Japan Science and Technology Agency (JST) as part of Adopting Sustainable Partnerships for Innovative Research Ecosystem (ASPIRE), Grant Number JPMJAP2520.

\end{document}